\documentclass[11pt]{aims}
\usepackage{amsmath, amssymb, mathrsfs}
  \usepackage{paralist}
  \usepackage{graphics} %% add this and next lines if pictures should be in esp format
  \usepackage{epsfig} %For pictures: screened artwork should be set up with an 85 or 100 line screen
\usepackage{graphicx}  \usepackage{epstopdf}%This is to transfer .eps figure to .pdf figure; please compile your paper using PDFLeTex or PDFTeXify.
 \usepackage[colorlinks=true]{hyperref}
 \hypersetup{urlcolor=blue, citecolor=red}
 \usepackage[toc,page]{appendix}
\usepackage{chngcntr}

\usepackage{titlesec}
\titleformat{\section}{\Large\bfseries}{\thesection.}{4pt}{}
\titleformat{\subsection}{\large\bfseries}{\thesection.\arabic{subsection}.}{4pt}{}
\titleformat{\subsubsection}{\bfseries}{\thesection.\arabic{subsection}.\arabic{subsubsection}.}{4pt}{}
\titleformat*{\paragraph}{\bfseries}
\titleformat*{\subparagraph}{\bfseries}
\usepackage[margin=1.2in]{geometry}

  \def\currentyear{}
   \def\currentmonth{}
    \def\ppages{}
     \def\DOI{}

\newtheorem{theorem}{Theorem}[section]

\newtheorem{lemma}[theorem]{Lemma}
\newtheorem{proposition}[theorem]{Proposition}
\theoremstyle{definition}
\newtheorem{definition}[theorem]{Definition}
\newtheorem{remark}[theorem]{Remark}

\newcommand{\Rb}{\mathbb{R}}
\newcommand{\Lc}{\mathcal{L}}
\newcommand{\Hc}{\mathcal{H}}
\newcommand{\Vc}{\mathcal{V}}
\newcommand{\Sc}{\mathcal{S}}

\newcommand{\Oc}{\mathcal{O}}

\newcommand{\Cc}{\mathcal{C}}
\newcommand{\Mc}{\mathcal{M}}

\newcommand{\Dc}{\mathcal{D}}

\numberwithin{equation}{section}

\title[Type I blowup solutions for non-variational semilinear parabolic systems] %use the shortened version of the full title
      {Construction and Stability of type I blowup solutions for non-variational semilinear parabolic systems}

\author[T. Ghoul, V. T. Nguyen, H. Zaag]{}

 \email[T. Ghoul]{teg6@nyu.edu}
 \email[V. T. Nguyen]{Tien.Nguyen@nyu.edu}
 \email[H. Zaag]{Hatem.Zaag@univ-paris13.fr}

\thanks{H. Zaag is supported by the ANR project ANA\'E ref. ANR-13-BS01-0010-03. \\ --------------------\\ 
\today}
\begin{document}
\maketitle

% Enter the first author's name and address:
\centerline{\scshape Tej-Eddine Ghoul$^\dagger$, Van Tien Nguyen$^\dagger$ and Hatem Zaag$^\ast$}
\medskip
{\footnotesize
 \centerline{$^\dagger$New York University in Abu Dhabi, P.O. Box 129188, Abu Dhabi, United Arab Emirates.}
  \centerline{$^\ast$Universit\'e Paris 13, Sorbonne Paris Cit\'e, LAGA, CNRS (UMR 7539), F-93430, Villetaneuse, France.}
}

\bigskip

\begin{abstract} We consider in this note the semilinear heat system 
$$\partial_t u = \Delta u + f(v), \quad \partial_t v = \mu\Delta v + g(u), \quad \mu > 0,$$
where the nonlinearity has no gradient structure taking of the particular form 
$$f(v) = v|v|^{p-1} \quad \textup{and}\quad g(u) = u|u|^{q-1} \quad  \textup{with} \quad p, q > 1, $$
or

$$f(v) = e^{pv}\quad \textup{and} \quad g(u) = e^{qu} \quad \textup{with} \quad p,q > 0.$$
We exhibit type I blowup solutions for this system and give a precise description of its blowup profiles. The method relies on two-step procedure: the reduction of the problem to a finite dimensional one via a spectral analysis, then solving the finite dimensional problem by a classical topological argument based on index theory. As a consequence of our technique, the constructed solutions are stable under a small perturbation of initial data. The results and the main arguments presented in this note can be found in our papers \cite{GNZihp18,GNZjde18}.  
\end{abstract}

\section{Introduction.}
In \cite{GNZihp18,GNZjde18}, we consider the semilinear parabolic system 
\begin{equation}\label{sys:uv}
\arraycolsep=1.4pt\def\arraystretch{1.6}
\left\{ \begin{array}{rl}
\partial_t u &= \;\;\Delta u + f(v), \\
\partial_t v &= \mu\Delta v + g(u), 
\end{array}
\right. \quad (x,t) \in \Rb^+ \times \Rb^N, 
\end{equation}
with $N \geq 1$, $\mu > 0$ and $\big(u(0),v(0)\big) = \big(u_0, v_0\big)$, where $(u,v)(t): x \in \Rb^N \to \Rb^2$ and the nonlinearity has no gradient structure taking of the particular form 
\begin{equation}\label{def:fg1}
f(v) = v|v|^{p-1} \quad \textup{and}\quad g(u) = u|u|^{q-1} \quad  \textup{with} \quad p, q > 1, 
\end{equation}
or
\begin{equation}\label{def:fg2}
f(v) = e^{pv}\quad \textup{and} \quad g(u) = e^{qu} \quad \textup{with} \quad p,q > 0.
\end{equation}
System \eqref{sys:uv} represents a simple model of a reaction-diffusion system describing heat propagation in a two-component combustible mixture and, as such, it has been the subject of intensive investigation from the last two decades (see \cite{SOUbook05}, \cite{ZZCna02} and references therein). We are here mainly interested in proving the existence and stability of finite time blowup solutions satisfying some prescribed asymptotic behavior. By finite time blowup, we mean that $T = T(u_0, v_0)$, the maximal existence  time of the classical solution $(u,v)$ of problem \eqref{sys:uv}, is finite, and the solution blows up in finite time $T$ in the sense that 
$$\lim_{t \to T}\big(\|u(t)\|_{L^\infty(\Rb^N)} + \|v(t)\|_{L^\infty(\Rb^N)} \big) = +\infty.$$ 
Moreover, a finite blowup solution $(u,v)$ of system \eqref{sys:uv} is called \textit{Type I} if there exists some positive constant $C$ such that 
\begin{equation}\label{def:TypeI}
\|u(t)\|_{L^\infty(\Rb^N)} \leq C\bar u(t), \quad \|v(t)\|_{L^\infty(\Rb^N)} \leq C \bar v(t),
\end{equation}
where $(\bar u, \bar v)(t)$ is the unique positive blowup solution of the ordinary differential system associated to \eqref{sys:uv}, namely that 
\begin{eqnarray*}
& \bar u(t) = \Gamma (T-t)^{-\alpha}, \quad \bar v(t) = \gamma(T- t)^{-\beta} & \quad \textup{for \eqref{def:fg1}},\\
& \bar u(t) = \ln\left[\big(p(T-t)\big)^\frac{-1}{q}\right], \quad \bar v(t) = \ln\left[\big(q(T-t)\big)^\frac{-1}{p}\right]& \quad \textup{for \eqref{def:fg2}},
\end{eqnarray*}
where $(\Gamma, \gamma)$ is determined by 
\begin{equation}\label{def:Gg}
\gamma^p = \alpha \Gamma, \quad \Gamma^q =  \gamma \beta, \quad \alpha = \frac{p + 1}{pq-1}, \quad \beta = \frac{q+1}{pq-1}.
\end{equation}
Otherwise, the blowup solution is of \textit{Type II}.

As for system  \eqref{sys:uv}-\eqref{def:fg1} with $\mu = 1$, the existence of finite time blowup solutions was derived by Friedman-Giga \cite{FGfsut87}, Escobedo-Herrero \cite{EHjde91} (see also \cite{EHpams91}, \cite{EHampa93}, etc). From Andreucci-Herrero-Vel\'azquez \cite{AHVihp97}, we know that estimate \eqref{def:TypeI} holds true if 
$$pq > 1, \quad q(pN - 2)_+ < N + 2 \quad \textup{or} \quad p(qN - 2)_+ < N + 2.$$
See also Caristi-Mitidieri \cite{CMjde94}, Deng \cite{Dzamp96}, Fila-Souple \cite{FSnodea01} for more results relative to estimate \eqref{def:TypeI}. Knowing that the solution exhibits Type I blowup, the authors of \cite{AHVihp97} were able to obtain more information about the asymptotic behavior of the solution near the singularity. Their results were later improved by Zaag \cite{Zcpam01}. When $\mu \ne 1$, much less result has been known, a part from Mahmoudi-Souplet-Tayachi \cite{MSTjde15} who establishes a single point blowup result that improves the one obtained in \cite{FGfsut87}. As for system \eqref{sys:uv} coupled with the nonlinearity \eqref{def:fg2}, the only known result is due to Souplet-Tayachi \cite{STna16} who adapted the technique developed in \cite{MSTjde15} to obtain the single point blowup result for a class of radially descreasing solutions. To our knowledge, there are no results concerning the asymptotic behavior, even for the equidiffusive case, i.e. $\mu = 1$. Also we recall that the study of the non-equidiffusive parabolic system \eqref{sys:uv} ($\mu$ may or may not equal to $1$) are in particular much more involved, both in terms of behavior of solutions and at the technical level.

In this note we exhibit \textit{Type I} blowup solutions for system \eqref{sys:uv} and give the first complete description of its asymptotic behavior. More precisely, we prove in \cite{GNZihp18} the following theorem.
\begin{theorem}[Type I blowup solutions for \eqref{sys:uv}-\eqref{def:fg1} and its asymptotic behavior, \cite{GNZihp18}] \label{theo:1}
Let $a \in \Rb^N$ and $T > 0$. There exist initial data $(u_0, v_0) \in L^\infty(\Rb^N) \times L^\infty(\Rb^N)$ for which system \eqref{sys:uv}-\eqref{def:fg1} has the unique solution $(u,v)$ defined on $\Rb^N \times [0, T)$ such that
\begin{itemize}
\item[(i)] The solution $(u,v)$ blows up in finite time $T$ at the only point $a$.
\item[(ii)] \textup{(Asymptotic profile)} There holds for all $t \in [0, T)$,
\begin{equation}\label{eq:asymuv1}
\left\|(T-t)^\alpha u(x,t) - \Phi^*(z)\right\|_{L^\infty(\Rb^N)} + \left\|(T-t)^\beta v(x,t) - \Psi^*(z)\right\|_{L^\infty(\Rb^N)} \leq \frac{C}{\sqrt{(T-t)}},
\end{equation}
where $z = \frac{x - a}{\sqrt{(T-t)|\log (T-t)|}}$ and the profiles $\Phi_0$ and $\Psi_0$ are explicitly given by
\begin{equation}\label{def:PhiPsi1}
\forall z \in \Rb^N, \quad \Phi^*(z) = \Gamma\big(1 + b |z|^2\big)^{-\alpha}, \quad \Psi^*(z) = \gamma\big(1 + b |z|^2\big)^{-\beta},
\end{equation} 
with $\Gamma, \gamma, \alpha, \beta$ being introduced in \eqref{def:Gg} and 
\begin{equation}
b = b(\mu, p,q) = \frac{(pq - 1)(2pq + p + q)}{4pq(p+1)(q+1)(\mu + 1)} > 0.
\end{equation}
\item[(iii)] \textup{(Final blowup profile)} For all $x \neq a$, $\big(u(x,t), v(x,t) \big) \to \big( u^*(x), v^*(x)\big) \in \Big[\Cc^2\big(\Rb^N \setminus \{0\} \big)\Big]^2$ with 
\begin{equation}\label{eq:finalpro}
u^*(x) \sim \Gamma\left(\frac{b|x- a|^2}{2 \big|\log |x-a|\big|} \right)^{-\frac{p + 1}{pq-1}}, \quad v^*(x) \sim \gamma\left(\frac{b|x- a|^2}{2 \big|\log |x-a|\big|} \right)^{-\frac{q + 1}{pq-1}}
\end{equation} 
as $|x - a| \to 0$.
\end{itemize}
\end{theorem}
\begin{remark} The asymptotic profile defined in \eqref{def:PhiPsi1} with $\mu = 1$ is the one among the classification result established in \cite{AHVihp97} (see also \cite{Zcpam01}). This is to say that we can construct Type I blowup solutions for \eqref{sys:uv}-\eqref{def:fg1} verifying the other asymptotic profiles described as in \cite{AHVihp97}. However, those constructions would be simpler than our considered case \eqref{eq:asymuv1} which involves some logarithmic correction to the blowup variable. 
\end{remark}
\begin{remark} The estimate \eqref{eq:finalpro} is sharp in comparison with the result established in \cite{MSTjde15} (see Theorem 1.3) where the authors could only obtain lower pointwise estimates without the logarithmic correction.
\end{remark}

As for system \eqref{sys:uv} coupled with the nonlinearity \eqref{def:fg2}, we study in the special affine space $\Hc_\alpha$ for some positive constant $\alpha$, 
\begin{equation*}
\Hc_\alpha = \big\{(u,v) \in (\bar \phi, \bar \psi) + L^\infty(\Rb^N) \times L^\infty(\Rb^N) \; \textup{where} \; q\bar \phi(x) = p \bar \psi(x) = - \ln(1 + \alpha|x|^2) \big\}, 
\end{equation*} 
and establish in \cite{GNZjde18} the following result:
\begin{theorem}[Type I blowup solutions for \eqref{sys:uv}-\eqref{def:fg2} and its asymptotic behavior, \cite{GNZjde18}] \label{theo:2} Let $a \in \Rb^N$ and $T > 0$. There exist initial data $(u_0, v_0) \in \Hc_\alpha$ for which system \eqref{sys:uv}-\eqref{def:fg2} has the unique solution $(u,v)$ defined on $\Rb^N \times [0, T)$ such that
\begin{itemize}
\item[(i)] The function $(e^{qu},e^{pv})$ blows up in finite time $T$ at the only point $a$.
\item[(ii)] \textup{(Asymptotic profile)} There holds for all $t \in [0, T)$,
\begin{equation}\label{eq:asym2}
\left\|(T-t)e^{qu(x,t)} - \Phi_*(z)\right\|_{L^\infty(\Rb^N)} + \left\|(T-t)e^{pv(x,t)} - \Psi_*(z)\right\|_{L^\infty(\Rb^N)} \leq \frac{C}{\sqrt{(T-t)}},
\end{equation}
where $z = \frac{x - a}{\sqrt{(T-t)|\log (T-t)|}}$ and the profiles $\Phi_*$ and $\Psi_*$ are explicitly given by
\begin{equation}\label{def:PhiPsi2}
\forall z \in \Rb^N, \quad  p\Phi_*(z) =  q\Psi_*(z) = \frac{1}{1 + b |z|^2} \quad \text{with} \;\;  b = \frac{1}{2(\mu + 1)}.
\end{equation} 
\item[(iii)] \textup{(Final blowup profile)} For all $x \neq a$, $\big(u(x,t), v(x,t) \big) \to \big( u_*(x), v_*(x)\big) \in \Big[\Cc^2\big(\Rb^N \setminus \{0\} \big)\Big]^2$ with 
\begin{equation}
u_*(x) \sim \frac 1 q \ln\left(\frac{2b}{p} \frac{\big|\log |x-a|\big|}{|x- a|^2} \right), \quad v_*(x) \sim \frac 1 p \ln\left(\frac{2b}{q} \frac{\big|\log |x-a|\big|}{|x- a|^2} \right) \quad \text{as}\; \; |x - a| \to 0.
 \end{equation} 
\end{itemize}
\end{theorem}

\begin{remark} We can construct Type I blowup solutions for \eqref{sys:uv}-\eqref{def:fg2} satisfying different blowup profiles that do not have logarithmic correction to the blowup variable described as in \eqref{eq:asym2}. This is to say that we can obtain an analogous classification result for Type I blowup solutions of \eqref{sys:uv}-\eqref{def:fg2} by adapting the technique of \cite{AHVihp97} with some more technical difficulties.
\end{remark}

The proof of Theorem \ref{theo:1} and Theorem \ref{theo:2} relies on two-step procedure:
\begin{itemize}
\item Reduction of an infinite dimensional problem to a finite dimensional one, through either the spectral analysis of the linearized operator around the expected profile or the energy-type estimate via the derivation of suitable Lyapunov functional. Note that the energy-type method breaks down for our problem because of the non gradient structure of the nonlinearity.
\item The control of the finite dimensional problem thanks to a classical topological argument based on index theory.
\end{itemize}
This two-step procedure has been successfully applied for various nonlinear evolution equations to construct both Type I and Type II blowup solutions. It was the case of the semilinear heat equation treated in \cite{BKnon94}, \cite{MZdm97}, \cite{NZens16} (see also \cite{NZsns16}, \cite{DNZtjm18} for the case of logarithmic perturbations, \cite{Breiumj90}, \cite{Brejde92} and \cite{GNZjde17} for the exponential source, \cite{NZcpde15} for the complex-valued case), the Ginzburg-Landau equation in \cite{MZjfa08}, \cite{NZarma18} (see also \cite{ZAAihn98} for an earlier work). It was also the nonlinear Schr\"odinger equation both in the mass critical \cite{MRgfa03,MRim04, MRam05, MRcmp05} and mass supercritical \cite{MRRcjm15} cases; the energy critical \cite{DKMcjm13}, \cite{HRapde12} and supercritical \cite{Car161} wave equation; the mass critical gKdV equation \cite{MMRam14, MMRasp15, MMRjems15}; the two dimensional Keller-Segel model \cite{RSma14}; the energy critical and supercritical geometric equations: the wave maps \cite{RRmihes12} and \cite{GINjde18}, the Schr\"odinger maps \cite{MRRim13} and the harmonic heat flow \cite{RScpam13, RSapde2014} and \cite{GINapde18}; the semilinear heat equation in the energy critical \cite{Sjfa12} and supercritical \cite{Car16} cases.\\

As a consequence of our technique, we obtained the following stability result.
\begin{theorem} The constructed solutions described in Theorem \ref{theo:1} and Theorem \ref{theo:2} are stable with respect to initial data. 
\end{theorem}
\begin{remark} The idea behind the stability result can be formally understood from the space-time and scaling invariance of the problem as follows: The linearized operator around the expected profile has two positive eigenvalues $\lambda_0 = 1, \lambda_1 = \frac{1}{2}$, a zero eigenvalue $\lambda_2 = 0$, then a an infinity discrete negative spectrum. From the analysis of stability of blowup problems, the component corresponding to $\lambda_0 = 1$ has the exponential growth $e^s$, which can be eliminated by a changing of the blowup time; similarly for the mode $\lambda_1 = \frac{1}{2}$ by a shifting of the blowup point; the neutral mode $\lambda_2 = 0$ usually has a polynomial growth and can be eliminated by using the scaling invariance of the problem. Since the remaining modes of the linearized operator corresponding to the negative spectrum decay exponentially, one derive the stability of the constructed solutions described in Theorem \ref{theo:1} and Theorem \ref{theo:2}. From the stability result, we expect that the blowup profiles \eqref{def:PhiPsi1} and \eqref{def:PhiPsi2} are generic, i.e. the other blowup profiles are unstable. In our opinion, this is a difficult open question whose a partly particular answer was given by Herrero-Vel\'azquez \cite{HVasnsp92} for the one dimensional semilinear heat equation.
\end{remark}

\section{A formal computation of the blowup profile.}
We brieftly recall in this section the formal approaches in \cite{GNZihp18,GNZjde18} to construct a suitable approximate blowup profile for system \eqref{sys:uv}. Similar approaches can be found in \cite{TZnor15, TZpre15}, \cite{GNZjde17}, \cite{NZarma18} and references therein. The method is based on matched asymptotic expansions which mainly replies on the spectral properties of the linearized operator around an expected profile.   \\

\paragraph{Similarity variables:} We perform the well known change of variables
\begin{eqnarray}
&\Phi(y,s) = (T-t)^\alpha u(x,t), \quad \Psi(y,s) = (T-t)^\beta v(x,t) &\quad \textup{for}\; \eqref{def:fg1},\label{def:sim1} \\
&\Phi(y,s) = (T-t) e^{qu(x,t)}, \quad \Psi(y,s) = (T-t) e^{pv(x,t)} &\quad \textup{for}\; \eqref{def:fg2},\label{def:sim2}
\end{eqnarray}
where $\alpha, \beta$ are introduced in \eqref{def:Gg} and 
\begin{equation*}
y = \frac{x}{\sqrt{T-t}}, \quad s = -\log(T-t). 
\end{equation*}
In this way, $(\Phi, \Psi)$ satisfies the new system
\begin{eqnarray}
\left\{\arraycolsep=1.6pt\def\arraystretch{2}
\begin{array}{ll}
\partial_s\Phi &= \Lc_1 \Phi - \alpha\Phi + |\Psi|^{p-1}\Psi,\\
\partial_s\Psi &= \Lc_\mu \Psi - \beta\Psi + |\Phi|^{p-1}\Phi,
\end{array} \right.  &\quad \textup{for \eqref{def:fg1},} \label{sys:PhiPsi1}\\
\left\{\arraycolsep=1.6pt\def\arraystretch{2}
\begin{array}{l}
\partial_s \Phi = \Lc_1 \Phi - \Phi - \dfrac{|\nabla \Phi|^2}{\Phi} + q \Phi \Psi,\\
\partial_s \Psi = \Lc_\mu \Psi - \Psi - \mu\dfrac{|\nabla \Psi|^2}{\Psi} + p \Phi \Psi,\end{array}\right. &\quad  \textup{for \eqref{def:fg2},} \label{sys:PhiPsi2}
\end{eqnarray}
where 
\begin{equation}\label{def:Leta}
\Lc_\eta f = \eta \Delta f - \frac{y}{2}\cdot \nabla f = \frac{\eta}{\rho_\eta} \nabla \cdot \big(\rho_\eta \nabla f\big) \quad \textup{with} \quad \eta \in \{1, \mu\},
\end{equation}
is the self-adjoint operator with respect to the Hilbert space $L^2_{\rho_\eta}(\Rb^N, \Rb)$ equipped with the inner product
$$\big<f,g\big>_{L^2_{\rho_\eta}} = \int_{\Rb^N}f(y) g(y) \rho_\eta(y) dy \quad \textup{with} \quad \rho_\eta(y) = \frac{1}{(4\pi)^{N/2}}e^{-\frac{|y|^2}{4\eta}}.$$

\paragraph{Linearized problem and spectral properties of the associated linearized operator:} Note that the nonzero constant solutions to systems \eqref{sys:PhiPsi1} and \eqref{sys:PhiPsi2} are $(\Gamma, \gamma)$ and $(1/p, 1/q)$ respectively. This suggests the linearization 
\begin{equation}\label{eq:linPhiPsicons}
\big(\bar \Phi, \bar \Psi\big) = \big(\Phi - \Gamma, \Psi - \gamma\big)\;\; \textup{for \eqref{def:fg1}} \quad \textup{and} \quad \big(\bar \Phi, \bar \Psi\big) = \big(\Phi - 1/p, \Psi - 1/q\big) \;\; \textup{for \eqref{def:fg2}}, 
\end{equation}
where $(\bar \Phi, \bar \Psi)$ solves the system 
\begin{equation}\label{sys:barPhiPsi}
i = 1, 2, \quad \partial_s \binom{\bar \Phi}{\bar \Psi} = \left(\Hc + \Mc_i\right)\binom{\bar \Phi}{\bar \Psi}  + \binom{Q_{i,1}}{Q_{i,2}},
\end{equation}
where $i = 1$ stands for the polynomial nonlinearity \eqref{def:fg1} and $i = 2$ for the exponential case \eqref{def:fg2}, $Q_{i, 1}$ and $Q_{i,2}$ are built to be quadratic, and the linear operator $\Hc$ and matrices $\Mc_i$'s are defined by 
\begin{equation}\label{def:Hc}
\Hc = \left(\begin{matrix}
\Lc_1 & 0\\ 0 &\Lc_\mu
\end{matrix} \right), \quad  \Mc_1 = \begin{pmatrix}
 -\alpha &\; p\gamma^{p-1}\\
q\Gamma^{q-1} &\; -\beta
\end{pmatrix}, \quad \Mc_2 = \left(\begin{matrix}
0 & q/p\\ p/q &0
\end{matrix} \right). 
\end{equation}
The following lemma gives the spectral properties of $\Hc + \Mc_i$. 
\begin{lemma}[Diagonalization of $\Hc + \Mc_i$] \label{lemm:diagonal}  For all $n \in \mathbb{N}$, there exist polynomials $f_n, g_n, \tilde{f}_n$ and $\tilde{g}_n$ of degree $n$ such that

\begin{equation}\label{eq:HMspec1}
\Big(\Hc+ \Mc_i\Big)\binom{f_n}{g_n} = \left(1 - \frac{n}{2}\right)\binom{f_n}{g_n}, \quad  \Big(\Hc+ \Mc_i\Big)\binom{\tilde{f}_n}{\tilde{g}_n} = \lambda^-_{i,n}\binom{\tilde{f}_n}{\tilde{g}_n},
\end{equation}
where
$$\lambda^-_{1,n} = -\left(\frac{n}{2} + \frac{(p+1)(q+1)}{pq-1}\right) , \quad \lambda^-_{2,n} = - \left(1 + \frac{n}{2}\right).$$
\end{lemma}
\begin{proof} See Lemma 3.2  in \cite{GNZihp18} for the polynomial case \eqref{def:fg1} and Lemma 2.2 in \cite{GNZjde18} for the exponential case \eqref{def:fg2}. The reader is kindly invited to have a look at precise formulas of the eigenfunctions as well as a proper definition of the projection according to these eigenmodes in those papers. \end{proof}

\paragraph{Inner expansion:} From Lemma \ref{lemm:diagonal}, we know that $\binom{f_n}{g_n}_{n \geq 3}$ and $\binom{\tilde f_n}{\tilde g_n}_{n \geq 0}$ correspond to negative eigenvalues of $\Hc + \Mc_i$, therefore, we may consider the following formal expansion under the radially symmetric assumption of the solution, 
\begin{equation}\label{eq:decomPhiPsibar}
\binom{\bar \Phi}{\bar \Psi}(y,s) = a_0(s) \binom{f_0}{g_0}(y) + a_2(s)\binom{f_2}{g_2}, 
\end{equation}
where $|a_0(s)| + |a_2(s)| \to 0$ as $s \to +\infty$. Plugging this ansatz into \eqref{sys:barPhiPsi} and projecting onto $\binom{f_k}{g_k}, k = 0, 2$ yields the ordinary differential system 
\begin{equation}\label{sys:a02}
\arraycolsep=1.4pt\def\arraystretch{2}
\left\{\begin{array}{ll}
a_0' &= a_0 + \Oc(|a_0|^2 + |a_2|^2),\\
a_2' &= c_{*} a_2^2 + \Oc(|a_2|^3 + |a_0 a_2| + |a_0|^3),
\end{array}\right.
\end{equation}  
where 
$$c_* = \frac{2pq + p + q}{4pq(p+1)(q+1)(\mu + 1)} \;\; \textup{for \eqref{def:fg1}} \quad \textup{and}\quad  c_* = 2pq(\mu + 1) \;\; \textup{for \eqref{def:fg2}}.$$ 
Assume that $|a_0(s)| = o(|a_2(s)|)$ as $s \to +\infty$, we get
$$a_2(s) = -\frac{1}{c_* s} + \Oc\left(\frac{\log s}{s^2}\right) \quad \textup{and} \quad  |a_0(s)| = \Oc\left(\frac{1}{s^2}\right) \quad \textup{as} \;\; s \to +\infty.$$
From \eqref{eq:decomPhiPsibar}, \eqref{eq:linPhiPsicons} and the definition of the eigenfuntion $\binom{f_2}{g_2}$, we end up with the asymptotic behavior
\begin{eqnarray}
\left\{\arraycolsep=1.6pt\def\arraystretch{2}
\begin{array}{ll}
\Phi(y,s) &= \Gamma\left[1 - \frac{p+1}{c_*}\frac{|y|^2}{s}  - \frac{2p(1 - \mu)}{c_* s} \right] +  \Oc\left(\frac{\log s}{s^2}\right),\\
\Psi(y,s) &= \gamma\left[1  - \frac{q+1}{c_*}\frac{|y|^2}{s} - \frac{2q(\mu - 1)}{c_* s} \right] + \Oc\left(\frac{\log s}{s^2}\right),
\end{array} \right.  &\quad \textup{for \eqref{sys:PhiPsi1},} \label{asy:PhiPsi1}\\
\left\{\arraycolsep=1.6pt\def\arraystretch{2}
\begin{array}{l}
\Phi(y,s) =  \frac{1}{p}\left[ 1  - \frac{pq}{c_*}\frac{|y|^2}{s} + \frac{2\mu pq}{c_*s}\right] + \Oc\left(\frac{\log s}{s^2}\right),\\
\Psi(y,s) = \frac{1}{q} \left[ 1 - \frac{pq}{c_*}\frac{|y^2|}{s} + \frac{2pq}{c_*s} \right] + \Oc\left(\frac{\log s}{s^2}\right),
\end{array}\right. &\quad  \textup{for \eqref{sys:PhiPsi2},} \label{asy:PhiPsi2}
\end{eqnarray}
where the convergence takes place in $L_{\rho_1}^2(\Rb^N) \times L^2_{\rho_\mu}(\Rb^N)$ as well as uniformly on compact sets by standard parabolic regularity.\\
\paragraph{Outer expansion:} These above asymptotic expansions provide a relevant blowup variable
\begin{equation*}
z = \frac{y}{\sqrt{s}} = \frac{x}{\sqrt{(T-t)|\log (T-t)|}}.
\end{equation*}
We then try to search an approximate solution to \eqref{sys:PhiPsi1} (respectively \eqref{sys:PhiPsi2}) of the form 
\begin{equation}
\binom{\Phi}{\Psi}(y,s) = \binom{\Phi_0}{\Psi_0}(z) + \frac{1}{s}\binom{\Phi_1}{\Phi_1}(z) + \cdots,
\end{equation}
Plugging this anzats to \eqref{sys:PhiPsi1} (respectively \eqref{sys:PhiPsi2}) yields the leading order system 
\begin{eqnarray}
-\frac{z}{2}\Phi_0'  - \alpha\Phi_0  + \Phi_0^p = 0, \quad -\frac{z}{2}\Psi_0'  - \beta\Psi_0  + \Psi_0^q = 0,& \quad \textup{for \eqref{sys:PhiPsi1}}, \label{sys:ode1}\\
-\frac{z}{2}\Phi_0'  - \Phi_0  + q\Phi_0 \Psi_0 = 0, \quad -\frac{z}{2}\Psi_0'  - \Psi_0  + p\Phi_0\Psi_0 = 0,& \quad \textup{for \eqref{sys:PhiPsi2}}, \label{sys:ode2}
\end{eqnarray}
subject to the initial condition 
$$\big(\Phi_0, \Psi_0\big)(0) = \big(\Gamma, \gamma\big) \;\; \textup{for \eqref{asy:PhiPsi1}} \quad \textup{and} \quad \big(\Phi_0, \Psi_0\big)(0) = \big(1/p, 1/q\big) \;\; \textup{for \eqref{asy:PhiPsi2}}.$$
The solutions of these system are explicitly given by 
\begin{eqnarray}
\Phi_0(z) = \frac{\Gamma}{(1 + b|z|^2)^{\alpha}}, \quad \Psi_0(z) = \frac{\gamma}{(1 + b|z|^2)^{\beta}}& \quad \textup{for \eqref{sys:ode1}}, \label{sol:ode1}\\
\Phi_0(z) = \frac{1}{p(1 + b|z|^2)}, \quad \Psi_0(z) = \frac{1}{q(1 + b|z|^2)} & \quad \textup{for \eqref{sys:ode2}}, \label{sol:ode2}
\end{eqnarray}
where $b > 0$ is an integration constant. By matching the asymptotic expansions \eqref{sol:ode1} with \eqref{asy:PhiPsi1} and \eqref{sol:ode2} with \eqref{asy:PhiPsi2}, we obtain precisely the value of the constant $b$ as stated in Theorems \ref{theo:1} and \ref{theo:2} respectively. 

In conclusion, we have formally derived the following approximate blowup profile:
\begin{eqnarray}
\left\{\arraycolsep=1.6pt\def\arraystretch{2}
\begin{array}{ll}
\Phi(y,s) & \sim \varphi(y,s) =  \Phi_0\left(\frac{y}{\sqrt s}\right)  - \frac{2\Gamma p(1 - \mu)}{c_* s},\\
\Psi(y,s) & \sim \psi(y,s) = \Psi_0\left(\frac{y}{\sqrt s}\right) - \frac{2\gamma q(\mu - 1)}{c_* s},
\end{array} \right.  &\quad \textup{for \eqref{sys:PhiPsi1},} \label{def:pro1}\\
\left\{\arraycolsep=1.6pt\def\arraystretch{2}
\begin{array}{ll}
\Phi(y,s) = & \sim \varphi(y,s) =  \Phi_0\left(\frac{y}{\sqrt s}\right)  + \frac{2\mu q}{c_*s},\\
\Psi(y,s) = & \sim \psi(y,s) = \Psi_0\left(\frac{y}{\sqrt s}\right) + \frac{2p}{c_*s},
\end{array}\right. &\quad  \textup{for \eqref{sys:PhiPsi2}.} \label{def:pro2}
\end{eqnarray}

\section{The existence proof without technical details.}

We present all main arguments of the existence proof without technical details for which we kindly refer the interested reader to our papers \cite{GNZihp18, GNZjde18}. We first deal with the polynomial case \eqref{def:fg1}, i.e. the proof of Theorem \ref{theo:1}, then the exponential case \eqref{def:fg2}, i.e. the proof of Theorem \ref{theo:2}, which is more delicate due to the presence of the terms $\frac{|\nabla \Phi|^2}{\Phi}$ and $\frac{|\nabla \Psi|^2}{\Psi}$ in the similarity variables setting (see \eqref{sys:PhiPsi2}).  

\subsection{The polynomial case \eqref{def:fg1}.}
This subsection is devoted to the proof of part $(ii)$ of Theorem \ref{theo:1}. Parts $(i)$ and $(iii)$ are consequences of part $(ii)$. The reader can find all details of the proof in \cite{GNZihp18}. 
\paragraph{Formulation of the problem:} In view of the similarity variables \eqref{def:sim1}, we see that constructing blowup solutions for \eqref{sys:uv} coupled with \eqref{def:fg1} satisfying the asymptotic behavior \eqref{eq:asymuv1} is equivalent to constructing for \eqref{sys:PhiPsi1} a global in time solution $(\Phi, \Psi)$ such that 
\begin{equation}\label{eq:goal1}
\sup_{y \in \Rb^N}\Big( \big|\Phi(y,s) - \Phi^*(y/\sqrt{s})\big| + \big|\Psi(y,s) - \Psi^*(y/\sqrt{s})\big|\Big) \to 0 \quad \textup{as} \;\; s \to +\infty,
\end{equation}
where $\Phi^*$ and $\Psi^*$ are the profiles defined in Theorem \ref{theo:1}. From the formal computation of an approximate blowup profile presented in the previous section, we linearize \eqref{sys:PhiPsi1} around $(\varphi, \psi)$ defined in \eqref{def:pro1} instead of $(\Phi^*, \Psi^*)$, namely that we introduce
\begin{equation}
\binom{\Lambda}{\Upsilon} = \binom{\Phi}{\Psi} - \binom{\varphi}{\psi},
\end{equation} 
which leads the linearized system
\begin{equation}\label{eq:LamUp}
\partial_s \binom{\Lambda}{\Upsilon} = \Big(\Hc + \Mc_1 + V(y,s)\Big)\binom{\Lambda}{\Upsilon} + \binom{F_1(\Upsilon, y,s)}{F_2(\Lambda, y,s)} + \binom{R_1(y,s)}{R_2(y,s)},
\end{equation}
where $\Hc$ and $\Mc_1$ are defined in \eqref{def:Hc}, 
\begin{equation}\label{def:Vys}
V(y,s) = \begin{pmatrix} 
0 & p\big(\psi^{p-1} - \gamma^{p-1}\big)\\ q\big(\varphi^{q-1} - \Gamma^{q-1}\big) &0
\end{pmatrix} \equiv \begin{pmatrix}
0 & V_1\\ V_2 & 0
\end{pmatrix},
\end{equation}
\begin{equation}\label{def:Bys}
\binom{F_1(\Upsilon, y,s)}{F_2(\Lambda, y,s)} = \binom{|\Upsilon + \psi|^{p-1}(\Upsilon + \psi) - \psi^p - p\psi^{p-1}\Upsilon}{|\Lambda + \varphi|^{q-1}(\Lambda + \varphi) - \varphi^q - q\varphi^{q-1}\Lambda},
\end{equation}
and 
\begin{equation}\label{def:Rys}
\binom{R_1(y,s)}{R_2(y,s)} = \binom{-\partial_s \varphi + \Delta \varphi - \frac{1}{2}y\cdot \nabla \varphi - \left(\frac{p+1}{pq-1}\right)\varphi + \psi^p}{-\partial_s \psi + \mu\Delta \psi - \frac{1}{2}y\cdot \nabla \psi - \left(\frac{q+1}{pq-1}\right)\psi + \varphi^q}.
\end{equation}
Our aim turns to construct for system \eqref{eq:LamUp} a global in time solution $(\Lambda, \Upsilon)$ verifying 
\begin{equation}\label{eq:goalLU}
\sup_{y \in \Rb^N} \Big(\big|\Lambda(y,s)\big| + \big|\Upsilon(y,s)\big|\Big) \to 0 \quad \textup{as}\;\; s \to +\infty. 
\end{equation}
Since the solution $(\Lambda, \Upsilon)$ goes to zero as $s \to +\infty$ and the nonlinear term $(F_1, F_2)$ is built to be quadratic and the error term $(R_1, R_2)$ is of the size $s^{-1}$, we see that the dynamics of \eqref{eq:LamUp} are strongly influenced by the linear part $\Hc + \Mc_1 + V$. Here the potential $V$ behaves differently as follows:\\
 -  Outer region, i.e. $|y| \gtrsim \sqrt{s}$: for all $\epsilon > 0$, there exists $K_\epsilon > 0$ and $s_\epsilon > 0$ such that 
 $$\sup_{|y| \geq K_\epsilon \sqrt{s}, s \geq s_\epsilon}|V(y,s)| \leq \epsilon.$$
From Lemma \ref{lemm:diagonal}, we see that the linear operator $\Hc + \Mc_1 + V$ behaves as one with fully negative spectrum in the outer region, which makes analysis in this region simpler.\\
- Inner region, i.e. $|y| \lesssim \sqrt{s}$: the potential $V$ is considered as a perturbation of the linear part $\Hc + \Mc_1$. 

Since the behavior of $V$ in the inner and outer regions is different, this suggests to consider the dynamics of \eqref{eq:LamUp} for $|y| \lesssim \sqrt{s}$ and $|y| \gtrsim \sqrt{s}$ separately. To this end, we introduce the cut-off function 
\begin{equation}\label{def:chi}
\chi(y,s) = \chi_0\left(\frac{|y|}{K\sqrt{s}}\right), \quad \chi_0 \in \Cc_0^\infty(\Rb^+, [0,1]), \quad \chi_0(r) = \left\{ \begin{array}{ll} 1 & \textup{for}\;\; r \in [0,1], \\
0 & \textup{for}\;\; r \geq 2,
\end{array} \right. 
\end{equation}

where $K$ is a positive constant to be fixed large enough. We then define
\begin{equation}\label{def:LUe}
\binom{\Lambda_e}{\Upsilon_e} = \big(1 - \chi(y,s)\big)\binom{\Lambda}{\Upsilon}
\end{equation}
and consider the decomposition 
\begin{equation}\label{def:decom}
\binom{\Lambda}{\Upsilon}(y,s) = \sum_{n \leq M} \left[ \theta_n(s) \binom{f_n}{g_n} + \tilde{\theta}_n \binom{\tilde{f}_n}{\tilde{g}_n} \right] + \binom{\Lambda_-}{\Upsilon_-}(y,s),
\end{equation}
where $\theta_n = \Pi_{n}\binom{\Lambda}{\Upsilon}$ and $\tilde \theta_n = \tilde \Pi_n \binom{\Lambda}{\Upsilon}$ with $\Pi_n$ and $\tilde{\Pi}_n$ being the projections onto the modes $\binom{f_n}{g_n}$ and $\binom{\tilde f_n}{\tilde g_n}$ respectively, and $\binom{\Lambda_-}{\Upsilon_-} = \Pi_{-,M}\binom{\Lambda}{\Upsilon}$ is called the infinite-dimensional part with $\Pi_{-,M}$ being the projector on the eigen-subspace corresponding the spectrum of $\Hc$ lower than $\frac{1 - M}{2}$. Note that the decomposition \ref{def:decom} is unique. \\

\paragraph{Preparation of initial data and Definition of the shrinking set:} Given $A > 1$  and $s_0 \geq e$, we consider the initial data for system \eqref{eq:LamUp} of the form
\begin{equation}\label{def:idata}
\binom{\Lambda}{\Upsilon}_{A, s_0, d_0, d_1}(y) = \frac{A}{s_0^2}\left[d_0\binom{f_0}{g_0} + d_1 \cdot \binom{f_1}{g_1}\right] \chi(y, s_0),
\end{equation}  
where $d_0  \in \Rb$ and $d_1 \in \Rb^N$ are parameters of the problem. Our aim is to show that for a fixed large constant $A$, then $s_0 = s_0(A)$ is fixed large as well, there exist $(d_0, d_1) \in \Rb^{1 + N}$ so that system \eqref{eq:LamUp} with initial data at $s = s_0$ given by \eqref{def:idata} has the unique  solution $(\Lambda, \Upsilon)$ satisfies \eqref{eq:goalLU}. More precisely, we will show that the solution $(\Lambda, \Upsilon)$ belongs to the following shrinking set:
\begin{definition}[Definition of a shrinking set] \label{def:VA} For all $A \geq 1$ and $s \geq e$, we defined $\Vc_A(s)$ as the set of all $(\Lambda,\Upsilon) \in L^\infty(\Rb^N) \times L^\infty(\Rb^N)$ such that 
$$|\theta_0(s)| \leq \frac{A}{s^2}, \quad |\theta_1(s)| \leq \frac{A}{s^2}, \quad |\theta_2(s)| \leq \frac{A^4 \log s}{s^2},$$
$$|\theta_j(s)|\leq \frac{A^j}{s^\frac{j+1}{2}},\quad |\tilde{\theta}_j(s)| \leq \frac{A^j}{s^\frac{j+1}{2}} \;\; \text{for}\;\; 3\leq j\leq M, \quad |\tilde \theta_i(s)| \leq \frac{A^{2}}{s^2}\;\; \text{for}\;\; i = 0, 1,2,$$
$$ \left\|\frac{\Lambda_-(y,s)}{1 + |y|^{M+1}} \right\|_{L^\infty(\mathbb{R}^N)}\leq \frac{A^{M+1}}{s^{\frac{M+2}{2}}},\quad  \left\|\frac{\Upsilon_-(y,s)}{1 + |y|^{M+1}} \right\|_{L^\infty(\mathbb{R}^N)} \leq \frac{A^{M+1}}{s^{\frac{M+2}{2}}},$$
$$\|\Lambda_e(s)\|_{L^\infty(\Rb^N)} \leq \frac{A^{M+2}}{\sqrt{s}},\quad  \|\Upsilon_e(s)\|_{L^\infty(\Rb^N)} \leq \frac{A^{M+2}}{\sqrt{s}},$$
where $\Lambda_e, \Upsilon_e$ are defined by \eqref{def:LUe},  $\Lambda_-, \Upsilon_-$, $\theta_n$, $\tilde \theta_n$ are defined as in decomposition \eqref{def:decom}.
\end{definition}
\begin{remark} \label{rem:1} We can check that if $\binom{\Lambda}{\Upsilon} \in \Vc_A(s)$ for $s \geq e$, then
\begin{equation}\label{eq:LUinVAest}
\|\Lambda(s)\|_{L^\infty(\Rb)} + \|\Upsilon(s)\|_{L^\infty(\Rb)} \leq \frac{CA^{M+2}}{\sqrt{s}},
\end{equation}
for some positive constant $C$, hence, estimate \eqref{eq:goalLU} is proved.
\end{remark}
In the following we make sure that the initial data \eqref{def:idata} belongs to $\Vc_A(s_0)$.
\begin{proposition}[Properties of initial data \eqref{def:idata}] \label{prop:properinti} For each $A \gg 1$, there exist $s_0(A) \gg 1$  and a cuboid $\Dc_{s_0} \subset [-A, A]^{1 + N}$ such that for all $(d_0, d_1) \in \Dc_{s_0}$, the following properties hold:
\begin{itemize}
\item[(i)] The initial data \eqref{def:idata} belongs to $\Vc_A(s_0)$ with strict inequalities except for the estimates of $\theta_{0}(s_0)$ and $\theta_1(s_0)$.
\item[(ii)] The map $\Theta: \Dc_{s_0} \to \Rb^{1 + N}$, defined as $\Theta(d_0, d_1) = (\theta_0(s_0), \theta_1(s_0))$, is linear, one to one from $\Dc_{s_0}$ to $[-As_0^{-2}, As_0^{-2}]^{1 + N}$, and maps $\partial \Dc_{s_0}$ into $\partial \big([-As_0^{-2}, As_0^{-2}]^{1 + N}\big)$. Moreover, the degree of $\Theta$ on the boundary is different from zero.
\end{itemize}
\end{proposition} 
\begin{proof} See Proposition 3.3 in \cite{GNZihp18}.
\end{proof}

\paragraph{Existence of a solution to \eqref{eq:LamUp} trapped in $\Vc_A(s)$:} From Remark \ref{rem:1}, we aim at proving the following.

\begin{proposition}[Existence of a solution of \eqref{eq:LamUp} trapped in $\Vc_A(s)$] \label{prop:goalVA} There exists $A_1$ such that for all $A \geq A_1$, there exists $s_{0,1}(A)$ such that for all $s_0 \geq s_{0,1}$, there exists $(d_0,d_1)$ such that if $\binom{\Lambda}{\Upsilon}$ is the solution of \eqref{eq:LamUp} with initial data at $s_0$ given by \eqref{def:idata}, then $\binom{\Lambda(s)}{\Upsilon(s)} \in \Vc_A(s)$ for all $s \geq s_0$.
\end{proposition} 
\begin{proof} For a fixed constant $A \gg 1$ and $s_0(A) \gg 1$, we note from the local Cauchy problem for system \eqref{sys:uv}-\eqref{def:fg1} in $L^\infty(\Rb^N) \times L^\infty(\Rb^N)$  that for each initial data \eqref{def:idata}, system \eqref{eq:LamUp} has a unique solution which stays in $\Vc_{A}(s)$ until some maximum time $s_* = s_*(d_0,d_1)$. If $ s_*(d_0,d_1)= +\infty$ for some $(d_0,d_1) \in \Dc_{s_0}$, then the proof is complete. Otherwise, we argue by contradiction and suppose that $s_*(d_0,d_1) < +\infty$ for any $(d_0,d_1) \in \Dc_{s_0}$. By continuity and the definition of $s_*$, we note that the solution at time $s_*$ is on the boundary of $\Vc_{A}(s_*)$. Thus, at least one of the inequalities in the definition of $\Vc_A(s_*)$ is an equality. In the following proposition, we show that this can happen only for the two components $\theta_0(s_*)$ and $\theta_1(s_*)$.

\begin{proposition}[Reduction to a finite dimensional problem] \label{prop:redu} Assume that $\binom{\Lambda}{\Upsilon}$ is a solution of  \eqref{eq:LamUp} with initial data at $s = s_0$ given by \eqref{def:idata} with $(d_0,d_1) \in \Dc_{s_0}$, and $\binom{\Lambda(s)}{\Upsilon(s)} \in \Vc_{A}(s)$ for all $s \in [s_0, s_1]$ for some $s_1 \geq s_0$ and $\binom{\Lambda(s_1)}{\Upsilon(s_1)} \in \partial \Vc_A(s_1)$, then 
\begin{itemize}
\item[(i)] $\big(\theta_0(s_1), \theta_1(s_1)\big) \in \partial \left( \left[-\frac{A}{s_1^2}, \frac{A}{s_1^2}\right]\right)^{1 + N}$.
\item[(ii)] There exists $\nu_0 > 0$ such that 
$$\forall \nu \in (0, \nu_0), \quad \binom{\Lambda(s_1 + \nu)}{\Upsilon(s_1 + \nu)} \not \in \Vc_A(s_1 + \nu). $$
\end{itemize}
\end{proposition}
\begin{proof} The proof of Proposition \ref{prop:redu} is a direct consequence of the dynamics of system \eqref{eq:LamUp}. The idea is to project system \eqref{eq:LamUp} on the different components of the decomposition \eqref{def:decom} and \eqref{def:LUe}. For all details of the proof, see Section 5.2 in \cite{GNZihp18}.
\end{proof}

From part $(i)$ of Proposition \ref{prop:redu}, we see that 
$$\big(\theta_0(s_*), \theta_1(s_*)\big) \in \partial \left( \left[-\frac{A}{s_*^2}, \frac{A}{s_*^2}\right]\right)^{1 + N}.$$
Hence, we may define the rescaled flow $\Theta$ at $s = s_*$ as follows:
\begin{align*}
\Theta: \Dc_{s_0} &\to \partial \big([-1,1]^{1 + N}\big)\\
(d_0,d_1) & \mapsto \frac{s_*^2}{A}\big(\theta_0, \theta_1\big)_{d_0,d_1}(s_*),
\end{align*}
which is continuous from part $(ii)$ of Proposition \ref{prop:redu}. On the other hand, from Proposition \ref{prop:properinti}, we have the strict inequalities for the other components for $(d_0,d_1) \in \partial \Dc_{s_0}$. Applying part $(ii)$ of Proposition \ref{prop:redu}, we see that $\binom{\Lambda(s)}{\Upsilon(s)}$ must leave $\Vc_A(s)$ at $s = s_0$, hence, $s_*(d_0,d_1) = s_0$. Recalling from part $(ii)$ of Proposition \ref{prop:properinti} that the degree of $\Theta$ on the boundary is different from zero. A contradiction then follows from the index theory. This concludes that there must exist $(d_0,d_1) \in \Dc_{s_0}$ such that for all $s \geq s_0$, $\binom{\Lambda(s)}{\Upsilon(s)} \in \Vc_A(s)$. This concludes the proof of Proposition \ref{prop:goalVA} as well as part $(ii)$ of Theorem \ref{theo:1}.
\end{proof}

\paragraph{Equivalence of the final blowup profile:} We present the main argument for the proof  of part $(iii)$ of Theorem \ref{theo:1}. For each $x_0 \neq 0$ with $|x_0| \ll 1$, we introduce for all $(\xi, \tau) \in \Rb \times \left[-\frac{t_0(x_0)}{T - t_0(x_0)},1 \right)$ the auxillary functions  
$$g(x_0, \xi, \tau) = (T-t_0(x_0))^\alpha u(x,t), \quad h(x_0, \xi, \tau) = (T-t_0(x_0))^\beta v(x,t),$$
where 
\begin{equation}\label{eq:defx0t_0}
x = x_0 + \xi\sqrt{T - t_0(x_0)}, \quad t = t_0(x_0) + \tau(T - t_0(x_0)),
\end{equation}\label{eq:deft0x0unique}
and $t_0(x_0)$ is uniquely determined by 
 \begin{equation}\label{eq:relx0at0}
|x_0| = K\sqrt{(T - t_0(x_0))|\log (T - t_0(x_0))|} \quad \textup{for a fixed constant $K \gg 1$.}
 \end{equation}
From the invariance of system \eqref{sys:uv}-\eqref{def:fg1} under the scaling, $(g(x_0, \xi, \tau), h(x_0, \xi, \tau))$ also satisfies \eqref{sys:uv}-\eqref{def:fg1}. From \eqref{eq:deft0x0unique}, \eqref{eq:defx0t_0} and the asymptotic behavior \eqref{eq:asymuv1}, we have
$$\sup_{|\xi| \leq 2 |\log (T-t_0(x_0))|^{1/4}} \left|g(x_0, \xi, 0) - \Phi^*(K)\right| \leq \frac{C}{|\log (T - t_0(x_0))|^{1/4}} \to 0,$$
and 
$$\sup_{|\xi| \leq 2 |\log (T-t_0(x_0))|^{1/4}} \left|h(x_0, \xi, 0) - \Psi^*(K)\right| \leq \frac{C}{|\log (T - t_0(x_0))|^{1/4}} \to 0,$$
as $|x_0| \to 0$. From the continuity with respect to initial data for system \eqref{sys:uv}-\eqref{def:fg1} associated to a space-localization in the ball $B(0, |\xi| < |\log (T-t_0(x_0))|^{1/4})$, we can show that 
$$\sup_{|\xi| \leq 2 |\log (T-t_0(x_0))|^{1/4}, 0 \leq \tau < 1} \left|g(x_0, \xi, 0) - \hat{g}_{K}(\tau)\right| \leq \epsilon(x_0) \to 0,$$
and 
$$\sup_{|\xi| \leq 2 |\log (T-t_0(x_0))|^{1/4}, 0 \leq \tau < 1} \left|h(x_0, \xi, 0) - \hat{h}_{K}(\tau)\right| \leq \epsilon(x_0) \to 0,$$
as $x_0 \to 0$, where 
$$\hat g_{K}(\tau) = \Gamma(1 - \tau + bK^2)^{-\alpha}, \quad \hat h_{K}(\tau) = \gamma(1 - \tau + bK^2)^{-\beta},$$
is the solution of system \eqref{sys:uv}-\eqref{def:fg1} with constant initial data $(\Phi^*(K), \Psi^*(K))$.

Making $\tau \to 1$ and using \eqref{eq:defx0t_0} yields
\begin{align*}
u^*(x_0) &= \lim_{t \to T}u(x,t) = (T-t_0(x_0))^{-\alpha}\lim_{\tau \to 1}g(x_0, 0, \tau) \sim (T-t_0(x_0))^{-\alpha}\hat g_{K}(1),\\
v^*(x_0) &= \lim_{t \to T}v(x,t) = (T-t_0(x_0))^{-\beta}\lim_{\tau \to 1}h(x_0, 0, \tau)\sim (T-t_0(x_0))^{-\beta}\hat h_{K}(1),
\end{align*}
as $|x_0| \to 0$. Using the relation \eqref{eq:relx0at0}, we obtain
$$|\log(T - t_0(x_0))| \sim 2\log |x_0|, \quad T -t_0(x_0) \sim \frac{|x_0|^2}{2K^2 |\log |x_0||} \quad \text{as}\quad |x_0| \to 0,$$
hence,
$$u^*(x_0) \sim \Gamma \left(\frac{b|x_0|^2}{2|\log |x_0||} \right)^{-\alpha}, \quad v^*(x_0) \sim \gamma \left(\frac{b|x_0|^2}{2|\log |x_0||} \right)^{-\beta},$$
as $|x_0| \to 0$. This concludes the proof of part $(iii)$ of Theorem \ref{theo:1}. Note that part $(iii)$ directly gives the single point blowup which is the conclusion of part $(i)$. This completes the proof of Theorem \ref{theo:1}. For the proof of Theorem \ref{theo:2}, we refer to \cite{GNZihp18}. 

\subsection{The exponential case \eqref{def:fg2}.} In this section we shall sketch those variants of the previous arguments which are required for the proof of Theorem \ref{theo:2}. All details of the proof can be found in \cite{GNZjde18}. The main difference between the two cases is the presence of the nonlinear gradient terms $\frac{|\nabla \Phi|^2}{\Phi}$ and $\frac{|\nabla \Psi|^2}{\Psi}$ in \eqref{sys:PhiPsi2} after making the change of variables \eqref{def:sim2}. In view of the approximate profile \eqref{def:pro2}, the control of these terms is delicate, in particular when the solution goes to zero in the intermediate zone. In order to treat them, we introduce a very careful control of the solution in a 3-fold shrinking set defined as follows: For $K_0 > 0$, $\epsilon_0 > 0$ and $t \in [0, T)$, we set 
\begin{align*}
\Dc_1(t) &= \left\{x\;  \Big\vert \; |x| \leq K_0 \sqrt{|\ln(T-t)|(T-t)} \right\}\\
&\quad \equiv \left\{x \;\big\vert \; |y| \leq K_0 \sqrt{s}\right\} \equiv \left\{ x \; \Big\vert \; |z| \leq K_0\right\},\\
\Dc_2(t) &= \left\{x \; \Big\vert \; \frac{K_0}{4} \sqrt{|\ln(T-t)|(T-t)} \leq |x| \leq \epsilon_0 \right\}\\
&\quad  \equiv \left\{x \; \Big\vert\; \frac{K_0}{4} \sqrt{s} \leq |y| \leq \epsilon_0e^{\frac{s}{2}}\right\} \equiv \left\{ x \; \Big \vert\; \frac{K_0}{4} \leq |z| \leq  \frac{\epsilon_0}{\sqrt s }e^{\frac s2} \right\},\\
\Dc_3(t) &= \left\{x\;  \Big\vert \; |x| \geq \frac{\epsilon_0}{4} \right\} \equiv \left\{x \;\big\vert \; |y| \geq \frac{\epsilon_0}{4}e^{\frac{s}{2}}\right\} \equiv \left\{ x \; \Big\vert \; |z| \geq \frac{\epsilon_0}{4 \sqrt s}e^{\frac{s}{2}} \right\}.
\end{align*}

- In the \textit{blowup region} $\Dc_1$, we linearize \eqref{asy:PhiPsi2} around the approximate profile \eqref{def:pro2}, namely that  $(\Lambda, \Upsilon) = (\Phi, \Psi) - (\varphi, \psi)$ solves the system 
\begin{equation}\label{eq:LamUp2}
\partial_s \binom{\Lambda}{\Upsilon} = \Big(\Hc + \Mc_2 + V(y,s)\Big)\binom{\Lambda}{\Upsilon} + \binom{q}{p}\Lambda \Upsilon + \binom{R_1}{R_2} + \binom{G_1}{G_2},
\end{equation}
where $\Hc$ and $\Mc_2$ are defined by \eqref{def:Hc}, 
\begin{equation}\label{def:Vys2}
V(y,s) = \begin{pmatrix} 
q\psi - 1 & \quad q\big(\phi - 1/p\big)\\ p\big(\psi - 1/q\big) & \quad p\phi - 1
\end{pmatrix} = \begin{pmatrix} V_1 &V_2 \\V_3 &V_4
\end{pmatrix},
\end{equation}
\begin{equation}\label{def:Gys2}
\binom{G_1}{G_2} = \binom{-|\nabla (\Lambda + \phi)|^2 (\Lambda + \phi)^{-1} + |\nabla \phi|^2\phi^{-1}}{-\mu|\nabla(\Upsilon + \psi)|^2(\Upsilon + \psi)^{-1} + \mu |\nabla \psi|^2 \psi^{-1}},
\end{equation}
and 
\begin{equation}\label{def:Rys2}
\binom{R_1}{R_2} = \binom{-\partial_s \phi + \Delta \phi - \frac{1}{2}y\cdot \nabla \phi - \phi + q\phi\psi - |\nabla \phi|^2\phi^{-1}}{-\partial_s \psi + \mu\Delta \psi - \frac{1}{2}y\cdot \nabla \psi - \psi + p\phi\psi - \mu |\nabla \psi|^2\psi^{-1}}.
\end{equation}
The analysis is similar as for the polynomial case according to the decomposition \eqref{def:decom} and  the definition \eqref{def:LUe}.\\

- In the intermediate region $\Dc_2$, we control $(u,v)$ by introducing the following auxillary functions $(\tilde u, \tilde v)$ defined for $x \ne 0$, 
\begin{equation}\label{def:uvtilde}
\arraycolsep=1.6pt\def\arraystretch{2}
\left\{\begin{array}{l} 
\tilde{u}(x, \xi, \tau) = \frac{1}{q}\ln \sigma(x) + u\Big(x + \xi\sqrt{\sigma(x)}, t(x) + \tau \sigma(x)\Big),\\
\tilde{v}(x, \xi, \tau) = \frac{1}{p}\ln \sigma(x) + v\Big(x + \xi\sqrt{\sigma(x)}, t(x) + \tau \sigma(x)\Big),
\end{array} \right.
\end{equation}
where  $t(x)$ is uniquely defined for $|x|$ sufficiently small by 
\begin{equation}\label{def:tx}
|x| = \frac{K_0}{4}\sqrt{\sigma(x)|\ln \sigma(x)|} \quad \text{with} \quad \sigma(x) = T - t(x).
\end{equation}
By the scaling invariance of the problem, we see that $(\tilde{u}, \tilde{v})$ also satisfies  system \eqref{sys:uv}-\eqref{def:fg2}. We prove that $(\tilde u,\tilde v)$ behaves for 
$$|\xi| \leq \alpha_0\sqrt{|\ln \sigma(x)|}\quad \text{and} \quad \tau \in \left[\frac{t_0  - t(x)}{\sigma(x)},1\right)$$
for some $t_0 < T$ and $\alpha_0 > 0$, like the solution of the ordinary differential system 
\begin{equation}
\partial_\tau \hat u = e^{p\hat v}, \quad \partial_\tau \hat v = e^{q\hat u},
\end{equation}
subject to the initial data 
$$\hat u(0) = -\frac{1}{q}\ln \left[p\left(1 + \frac{K_0^2/16}{2(\mu + 1)}\right)\right], \quad \hat v(0) = -\frac{1}{p}\ln \left[q\left(1 + \frac{K_0^2/16}{2(\mu + 1)}\right)\right].$$
The explicit solution is given by 
\begin{equation}\label{def:solUc}
\hat u(\tau) = -\frac{1}{q}\ln \left[p\left(1 - \tau + \frac{K_0^2/16}{2(\mu + 1)}\right)\right], \quad \hat v(\tau) =-\frac{1}{p}\ln \left[q\left(1 - \tau + \frac{K_0^2/16}{2(\mu + 1)}\right)\right]. 
\end{equation}
The analysis in $\Dc_2$ directly yields the conclusion of part $(iii)$ of Theorem \ref{theo:2}. \\

- In $\Dc_3$, we directly control $(u,v)$ by using the local in time well-posedness of the Cauchy problem for system \eqref{sys:uv}.\\

The following definition of the shrinking set to trap the solution is the crucial difference in comparison with the existence proof for the polynomial case.
\begin{definition}[Definition of a shrinking set] \label{def:St} For all $t_0 < T$, $K_0 > 0$, $\epsilon_0 > 0$, $\alpha_0 > 0$, $A > 0$, $\delta_0 > 0$, $\eta_0 > 0$, $C_0 > 0$, for all $t \in [t_0,T)$, we define  $\;\Sc(t_0, K_0, \epsilon_0, \alpha_0, A, \delta_0, \eta_0, C_0,  t)$ being the set of all functions $(u,v)$ such that
\begin{itemize}
\item[(i)] \textit{(Control in $\Dc_1$)}  $\quad \binom{\Lambda(s)}{\Upsilon(s)} \in \Vc_A(s)$, where $\Vc_A(s)$ is introduced in Definition \ref{def:VA}.
\item[(ii)] \textit{(Control in $\Dc_2$)} For all $|x| \in \left[\frac{K_0}{4}\sqrt{|\ln(T-t)|(T-t)}, \epsilon_0\right]$, $\tau = \tau(x,t) = \frac{t - t(x)}{\sigma(x)}$ and $|\xi| \leq \alpha_0 \sqrt{\ln \sigma(x)}$, 
\begin{align*}
\left|\tilde u(x,\xi, \tau) - \hat u(\tau)\right| \leq \delta_0, \quad |\nabla_\xi \tilde u(x, \xi, \tau)| \leq \frac{C_0}{\sqrt{|\ln \sigma(x)|}}, \\
\left|\tilde v(x,\xi, \tau) - \hat v(\tau)\right| \leq \delta_0, \quad |\nabla_\xi \tilde v(x, \xi, \tau)| \leq \frac{C_0}{\sqrt{|\ln \sigma(x)|}}, 
\end{align*}
where $\tilde u, \tilde{v}$, $\hat u$, $\hat v$, $t(x)$ and $\sigma(x)$ are defined in \eqref{def:uvtilde}, \eqref{def:solUc} and \eqref{def:tx} respectively.
\item[(iii)] \textit{(Control in $\Dc_3$)} For all $|x| \geq \frac{\epsilon_0}{4}$, 
\begin{align*}
|\nabla_x ^i u(x,t) - \nabla_x ^i u(x, t_0)| \leq \eta_0 \quad \text{and} \quad |\nabla_x ^i v(x,t) - \nabla_x ^i v(x, t_0)| \leq \eta_0 \quad \text{for} \;\; i = 0,1.
\end{align*}
\end{itemize}
\end{definition}
\begin{remark} In comparison with Definition \ref{def:VA}, the shrinking set $\Sc$ has additional estimates in the domains $\Dc_2$ and $\Dc_3$. These estimates are crucially needed to achieve the control of the nonlinear gradient term $\binom{G_1}{G_2}$ appearing in \eqref{eq:LamUp2}. 
\end{remark}

After defining the shrinking set $\Sc$ to trap the solution, we need a suitable initial data for \eqref{eq:LamUp2} so that the corresponding solution gradually belongs to $\Sc(t)$ for all $t \in [t_0, T)$. To this end, we consider the following functions depending on $(N+1)$ parameters $(d_0, d_1) \in \Rb^{1 + N}$: 
\begin{align}
\binom{qu}{pv}_{d_0,d_1}(x,t_0) &= \binom{\hat u_*(x)}{\hat v_*(x)}\Big(1 - \chi_1(x,t_0)\Big) + \left\{\binom{1}{1}s_0 +\ln\left[\binom{\phi}{\psi}(y_0,s_0)\right]\right\}  \chi_1(x,t_0) \nonumber\\
& + \ln \left\{\left(d_0\binom{f_0(y_0)}{g_0(y_0)} + d_1. \binom{f_1(y_0)}{g_1(y_0)}\right)\frac{A^2}{s_0^2}\chi(16y_0, s_0)\right\}  \chi_1(x,t_0), \label{def:uvt0}
\end{align}
where $s_0 = -\ln(T-t_0)$, $y_0 = x e^{\frac{s_0}{2}}$, $\phi$ and $\psi$ are defined by \eqref{def:pro2}, $\binom{f_0}{g_0}$ and $\binom{f_1}{g_1}$ are the eigenfunctions introduced in Lemma \ref{lemm:diagonal}, $\chi$ is the cut-off function defined by \eqref{def:chi},
$$\chi_1(x,t_0) = \chi_0 \left(\frac{|x|}{|\ln(T-t_0)| \sqrt{T-t_0}}\right) = \chi_0\left(\frac{y_0}{s_0}\right),$$
and $(\hat u_*, \hat v_*) \in \Cc^\infty(\Rb^N \ \{0\}) \times \Cc^\infty(\Rb^N \setminus \{0\})$ is defined by 
\begin{equation*}\label{def:ustar}
\hat u_*(x) = \left\{\begin{array}{ll}
\ln\left(\frac{4(\mu + 1)|\ln |x||}{p|x|^2}\right) &\quad \text{for}\quad |x| \leq C(a),\\
 -\ln\left(1 + a |x|^2\right) &\quad \text{for} \quad |x|\geq 1,
\end{array}
 \right.
\end{equation*}
\begin{equation*}\label{def:vstar}
\hat v_*(x) = \left\{\begin{array}{ll}
\ln\left(\frac{4(\mu + 1)|\ln |x||}{q|x|^2}\right) &\quad \text{for}\quad |x| \leq C(a),\\
 -\ln\left(1 + a |x|^2\right) &\quad \text{for} \quad |x|\geq 1.
\end{array}
 \right.
\end{equation*} 

After having a proper definition of initial data and the shrinking set, the remaining step is to show that there exists $(d_0, d_1) \in \Rb^{1 + N}$ such that system \eqref{sys:uv}-\eqref{def:fg2} with initial data \eqref{def:uvt0} has a unique solution $(u,v) \in \Sc(t)$ for all $t \in [t_0, T)$. The main argument of this step is exactly the same as for the polynomial case, i.e. the proof of Proposition \ref{prop:goalVA}. We refer the interested reader on the reduction to a finite dimensional problem to Section 4 in \cite{GNZjde18} for all details. This concludes the proof of Theorem \ref{theo:2}.

%\bibliographystyle{plainnat}
%\bibliography{/Volumes/Data/Work/mybib} %for Mac

\def\cprime{$'$}

\end{document}